\documentclass[a4paper,12pt]{article}

\usepackage[T1]{fontenc}
\usepackage{lmodern}
\usepackage{microtype}
\usepackage{xcolor}
\usepackage{amsmath,amssymb,amsthm,mathtools}
\usepackage{geometry}
\usepackage{graphicx}
\numberwithin{equation}{section}

\usepackage{enumitem}
\setlist[itemize]{leftmargin=*, itemsep=2pt, topsep=2pt}

\usepackage{hyperref}
\hypersetup{
  colorlinks=true,
  linkcolor=blue,
  citecolor=blue,
  urlcolor=blue
}

\theoremstyle{plain}
\newtheorem{theorem}{Theorem}[section]
\newtheorem{proposition}[theorem]{Proposition}
\newtheorem{lemma}[theorem]{Lemma}
\newtheorem{corollary}[theorem]{Corollary}

\theoremstyle{definition}
\newtheorem{definition}[theorem]{Definition}

\theoremstyle{remark}
\newtheorem{remark}[theorem]{Remark}

\newcommand{\R}{\mathbb{R}}

\newcommand{\sgn}{\mathrm{sgn}}

\DeclareMathOperator{\sign}{sgn}

\title{Global boundary stabilization of 1d systems of scalar conservation laws}

\author{Georges Bastin\thanks{Department of Mathematical Engineering, ICTEAM, UCLouvain, Louvain-La-Neuve, Belgium.   (georges.bastin@uclouvain.be)} , Jean-Michel Coron\thanks{Laboratoire Jacques-Louis Lions, Sorbonne Universit\'{e}, Universit\'{e} de Paris, CNRS, INRIA, \'{e}quipe Cage, Paris, France. (jean-michel.coron@sorbonne-universite.fr)} \,\, and Amaury Hayat\thanks{CERMICS, Ecole des Ponts ParisTech, Champs-sur-Marne, France. (amaury.hayat@enpc.fr)}}
\date{\empty}

\begin{document}
\maketitle

\begin{abstract}
We study a system of several one-dimensional scalar conservation laws coupled through boundary feedback conditions that combine physical boundary constraints with static feedback control laws. Our first contribution establishes the well-posedness of the system in the space of $L^{\infty}$ entropy solutions.
Our second contribution provides a set of sufficient dissipative conditions on the boundary coupling that ensure global exponential stability in $L^1$ and $L^\infty$ norms.



\end{abstract}
\tableofcontents

\section{Introduction}

Boundary feedback stabilization of hyperbolic balance laws is a classical theme in control theory, motivated by applications where actuation and sensing are naturally available at the boundaries: open-channel flow, gas pipelines, traffic flow, and more generally transport phenomena on networks.
In the smooth (classical) regime, a large body of work relates exponential stability to \emph{dissipative} boundary conditions, typically expressed through a contraction property of the boundary map in suitable norms. Reference can be made to the pioneering work of Greenberg--Li \cite{1984-Greenberg-Li-JDE}, Qin \cite{1985-Qin-CAM}, Zhao \cite{1986-Zhao-thesis}, and Li \cite{1994-Li-book} as well as to the modern synthesis and extensions in \cite{2015-Bastin-Coron-SICON,2016-Bastin-Coron-book,HayatC12019}. In parallel, an approach through time-delay systems allows for sharp dissipativity conditions in $W^{2,p}$-norm \cite{CoronNguyen2015}.

In these works the solution of the hyperbolic balance laws  is required to be classical and the stability is usually studied in either $C^{1}$ or $H^{2}$ norm.
In many applications, however, discontinuities may occur in finite time, and the relevant notion of solution is that of \emph{entropy solutions}.
For scalar conservation laws, entropy well-posedness in $\mathbb{R}$ goes back to Kru\v{z}kov \cite{Kruzhkov1970}, and the initial--boundary value problem requires a careful entropy formulation of boundary conditions.
The Bardos--Le Roux--N\'ed\'elec (BLN) theory \cite{BLN79} provides the canonical admissibility condition at the boundary; alternative and complementary viewpoints include the viscosity/Riemann-problem approach of Dubois--LeFloch \cite{DuboisLeFloch1988} and subsequent developments in the general boundary-value framework \cite{AS15}.
A crucial technical ingredient to study boundary feedback stabilization is the existence of (strong) boundary traces for bounded entropy solutions: this is by now well established in broad generality, notably through the works of Kwon--Vasseur and Panov \cite{KV07,Panov07} and, for bounded domains with Dirichlet-type data, through the analysis of Coclite--Karlsen--Kwon \cite{CKK09}.
We also refer to standard monographs for background on hyperbolic conservation laws, BV estimates, and front-tracking/semigroup methods; see e.g.\ \cite{Bressan2000,HoldenRisebro2002}.

On the control side, stability and boundary stabilization results for non-classical solutions are comparatively more recent. For scalar laws, we can mention the asymptotic stabilization of entropy solutions (in the case of a convex flux) in \cite{Perrollaz2013}
and \cite{BlandinEtAl2017}, and the stabilization of a shock steady-state for the Burgers equation in \cite{burgers}. For systems of scalar conservation laws coupled by boundary interconnections, the semi-global exponential stability in BV norm using saturated controls is obtained in \cite{dus2021}. For $2\times 2$ systems of conservation laws, the local stabilization in BV norms is addressed in  \cite{CoronBV} while the local stabilization of $H^{2}$-solutions with a single shock is addressed in \cite{bastin2019boundary}.

In the present paper, we consider the following \textit{system of one-dimensional scalar conservation laws}
\begin{gather}
  u_t + (f(u))_x = 0, \qquad  t \in (0,+\infty), \; x \in (0,1), \label{eq:sys}\\
  u(t,0) = G(u(t,1)), \label{eq:feedback}
\end{gather}
where $u=(u_1,\dots,u_n)^\top$, the flux $f$ belongs to $C^{2}(\R^{n};\mathbb{R}^{n})$, and the map $G:\R^{n}\to\R^{n}$ is globally Lipschitz with $G(0)=0$.

We assume that the flux is diagonal i.e. $f(u)=(f_1(u_1),\dots,f_n(u_n))^\top$, which implies that \eqref{eq:sys} represents $n$ scalar conservation laws
\begin{equation}
	\partial_t u_i + \partial_x f_i(u_i) = 0, \quad i \in \{1, \dots ,n\}
\end{equation}
which are decoupled on the interior spatial domain $x \in (0,1)$ but coupled  through the boundary conditions \eqref{eq:feedback}.  We also assume that the characteristic velocities are strictly positive:
\begin{equation} \label{eq:pos}
\exists \, a	 > 0 \quad \text{such that} \quad f'_{i}(s)\ge a, \quad \forall s\in \R,\;\forall i\in\{1,...,n\}.
\end{equation}
This assumption is adopted throughout our article and, in particular, in each of the theorems even though it is not explicitly mentioned.

The coupling of the $n$ scalar conservation laws is induced by the boundary condition \eqref{eq:feedback} with the function $G$ representing a combination of physical boundary constraints and potential static feedback control laws.
Such couplings arise naturally in networked transport models as shown for instance in \cite{2007-Bastin-et-al-NHM} and \cite[Section 4.2]{2016-Bastin-Coron-book} with an example of ramp-metering control in road traffic networks.
It should also be noted that the positivity condition \eqref{eq:pos} on the direction of the characteristics  allows to consider that the system \eqref{eq:sys}--\eqref{eq:feedback} is a closed-loop interconnection of two causal input-output systems as represented in Figure \ref{fig-CL}.
\begin{figure}[hbt]
\centerline{ \includegraphics[width=10cm]{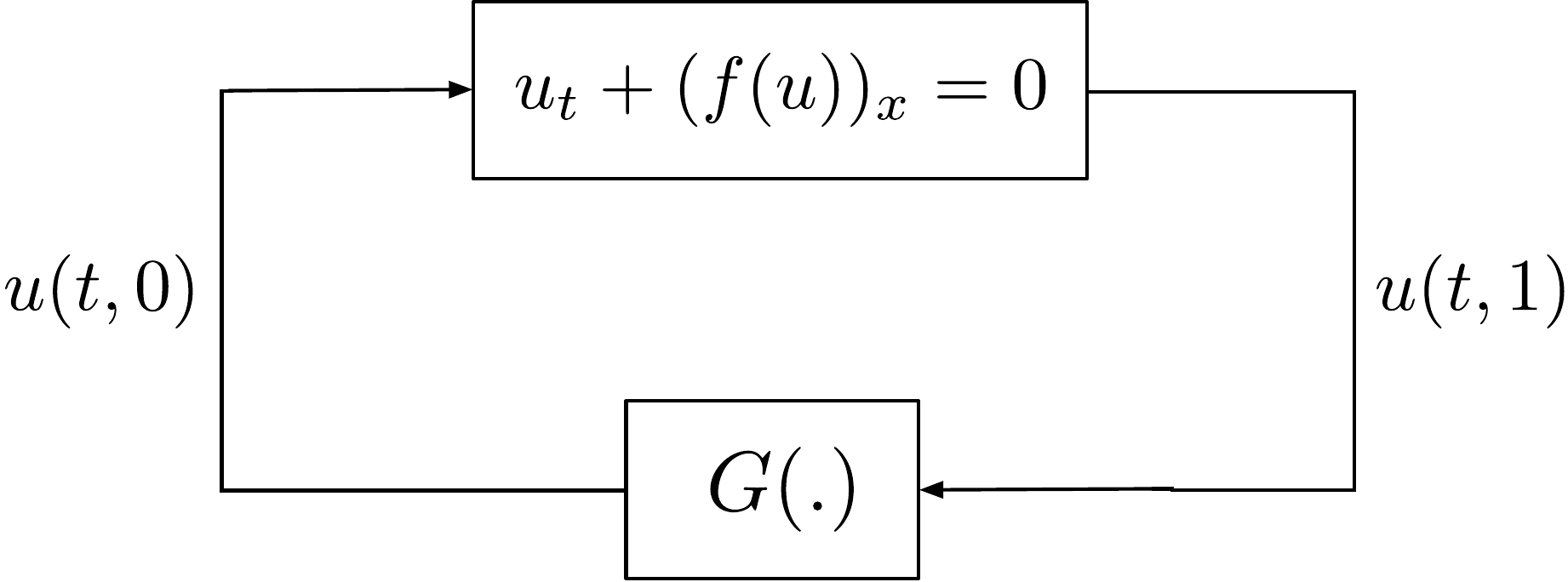}}
  \caption{A closed-loop interconnection of two causal input-output systems.}
  \label{fig-CL}
\end{figure}

\medskip
\noindent\textbf{Well-posedness with a nonlocal feedback.}
Our first contribution (Theorem~\ref{thm:main}) establishes global  well-posedness of the system \eqref{eq:sys}--\eqref{eq:feedback}, under a global Lipschitz assumption on $f$: it is shown that the system has a unique entropy solution with $u(t,0)$ and $u(t,1)$ denoting, respectively, the incoming and outgoing strong boundary traces of the solution that are shown to exist. The main difficulty to prove the theorem is that the boundary condition \eqref{eq:feedback} is non-local in space. However, since scalar conservation laws are delay-type systems, the outgoing trace $u(t,1)$ is independent, over sufficiently small time intervals, of the most recent values of the incoming signal $u(t,0)$. This allows to construct an iterative approach for proving the theorem where, for each successive time slab, the outgoing signal $u(t,1)$ is first computed as the output of an open-loop system which is known to be well-posed (in the BLN framework \cite{BLN79,CKK09,AS15} together with strong traces \cite{KV07,Panov07,CKK09}), and then injected in the feedback loop through the function $G$ to initiate the next step.
 In particular, with this approach, the non local boundary condition \eqref{eq:feedback} can be enforced without any smallness or contraction argument in the entropy setting.

\medskip
\noindent\textbf{Global exponential stability in $L^1$ and $L^\infty$.}
Our second contribution is a set of sufficient conditions on $G$ guaranteeing global exponential stability in the $L^1$ and $L^\infty$ norms respectively. Note that global exponential stability is usually impossible when using norms which are too regular (such as $C^{1}$ or $H^{2}$ norms). For this reason, almost all existing results deal with \textit{local} exponential stability \cite{2015-Bastin-Coron-SICON,2016-Bastin-Coron-book,HayatC12019,1985-Qin-CAM,CoronNguyen2015}. However, note that global exponential stability has been established for some particular semilinear systems in \cite{Hayat2021GlobalL2}.

For the $L^1$-norm, we introduce a weighted entropy/Lyapunov functional inspired by entropy-based network analyses \cite{2007-Bastin-et-al-NHM} and we derive a dissipation inequality in which the boundary term naturally involves $f_i$ (Theorem~\ref{thm:expstab}).
This yields a flux-dependent stability criterion, which is qualitatively different from the classical smooth theory where the Jacobian $G'(0)$ and matrix norms dominate the condition \cite{2016-Bastin-Coron-book,CoronNguyen2015}.
We then identify an important subclass (notably concave fluxes) for which the criterion reduces to a flux-independent weighted contraction property of $G$ for the $|\cdot|_1$-norm on $\R^n$ (Theorem~\ref{thm:fconvex-concave}).

For the $L^\infty$-norm, we prove global exponential stability under a weighted $\ell^\infty$ contraction of $G$ (Theorem~\ref{thm:expstab-infty}), in the spirit of dissipative boundary conditions for hyperbolic systems \cite{2015-Bastin-Coron-SICON,2016-Bastin-Coron-book}.
A notable feature is that, unlike the global well-posedness, the $L^\infty$ stability does \emph{not} require global Lipschitzness of $f$: the boundary contraction prevents amplitude growth along successive traversals of the domain and allows global-in-time control of the solution through a truncation argument.

\medskip
\noindent\textbf{Local exponential stability in $L^\infty$.}
The global exponential stability result requires a somewhat strong assumption on the $f_{i}$ in that $f_i'(s)\ge a>0$. Nevertheless, a consequence of our result is that one can remove this constraint if we only look at the local exponential stability (see Corollary \ref{cor:linear}).

\medskip
\noindent\textbf{Organization of the paper.}
Section~\ref{sec:entropy} recalls the entropy formulation and strong trace properties.
Section~\ref{sec:closed} is devoted to proving the well-posedness of the closed-loop system. The argument relies on the well-posedness and stability estimates for the open-loop system (Theorem~\ref{thm:open} and Proposition~\ref{prop:mp} below), which are based on \cite{CKK09}, as well as on the finite propagation speed of information (see Lemma~\ref{lem:outflow_indep} below). This property allows us to treat the closed-loop system as a sequence of open-loop systems, enabling the iterative application of the open-loop results. Finally, Sections~\ref{sec:proof-stab-L1} and~\ref{sec:Linfty} develop Lyapunov arguments yielding global exponential stability in $L^1$ and $L^\infty$ under the proposed dissipativity conditions on $G$. Theorems~\ref{thm:expstab} and~\ref{thm:fconvex-concave} are proved in Section~\ref{sec:proof-stab-L1}, and Theorem~\ref{thm:expstab-infty} is proved in Section~\ref{sec:Linfty}.

\section{Main results}

Let $T>0$, our functional framework is
\begin{equation}
\label{regularity-u}
  u \in C([0,T];L^1(0,1))\cap L^\infty((0,T)\times(0,1)).
\end{equation}
In \eqref{regularity-u} and in what follows, we use the abbreviations $L^{1}(0,1) = L^{1}((0,1);\mathbb{R}^{n})$ and $L^{\infty}((0,1)\times(0,T)) = L^{\infty}((0,1)\times(0,T);\mathbb{R}^{n})$. Similarly, $L^{p}(0,1) = L^{p}((0,1);\mathbb{R}^{n})$, $L^p(0,T)=L^{p}((0,T);\mathbb{R}^{n})$, etc.

Our first result is:

\begin{theorem}[Well-posedness of the closed-loop system]
\label{thm:main}
Let $T>0$ and assume $u_0\in L^\infty(0,1)$. If $f$ is globally Lipschitz,
then there exists a unique entropy solution $u$ of \eqref{eq:sys}--\eqref{eq:feedback} on $(0,T)\times(0,1)$ (in the sense of Definition \ref{def:kru}) with initial condition
\begin{equation}\label{eq:ic}
  u(0,x)=u_0(x) \qquad \text{for a.e. } x\in(0,1),
\end{equation}
 such that:
\begin{itemize}
  \item $u\in C([0,T];L^1(0,1))\cap L^\infty((0,T)\times(0,1))$;
  \item $u$ admits strong traces $u(\cdot,0),u(\cdot,1)\in L^\infty(0,T)$ which satisfy the boundary condition \eqref{eq:feedback} a.e.\ on $(0,T)$.
\end{itemize}
Moreover, $u_{0}\mapsto u$ is continuous from $L^{1}(0,1)\cap L^{\infty}(0,1)$ endowed with the $L^{1}$ norm to $C^{0}([0,T];L^{1}(0,1))$.
\end{theorem}

\begin{remark}
Note that the global Lipschitz assumption on $f$ is necessary in the absence of additional information on $G$. Otherwise, even in the scalar case $n=1$, blow-up may occur. For instance, consider $G(U)=2U$ and $f(U)=U^2/2$ for $U \ge 1$. Let $T>0$. Subject to the initial data
\begin{equation}
\label{def:sol-Burgers-init}
u_0(x) = \frac{2-x}{T},
\end{equation}
the corresponding solution to \eqref{eq:sys}, \eqref{eq:feedback}, and \eqref{eq:ic} is given by
\begin{equation}
\label{def:sol-Burgers}
u(t,x) = \frac{2-x}{T-t}\geq 1, \quad t\in (0,T), \quad x \in [0,1]
\end{equation}
and it blows up in time $T$.
\end{remark}

The proof of Theorem \ref{thm:main} (Section~\ref{sec:closed}) will use two ingredients:
\begin{itemize}
\item An analysis of the (open-loop) system for the initial boundary value problem \eqref{eq:sys} with boundary conditions of the form
\begin{equation}
\label{eq:open-bound}
u(t,0) = w(t),
\end{equation}
where $w\in L^{\infty}(0,T)$ is given (Theorem~\ref{thm:open} and Proposition~\ref{prop:mp} below).
\item A finite speed of propagation/delay lemma (Lemma \ref{lem:outflow_indep}) showing that,  because the outgoing trace $u(t,1)$ at time $t$
is not influenced by the most recent values of the incoming signal $u(t,0)$, the system \eqref{eq:sys} of scalar conservation laws can be considered as
 an open-loop system over a sufficiently short time interval.
\end{itemize}
Our next contribution is to give conditions on the map $G$ which guarantee the global exponential stability of systems of scalar conservation laws of the form \eqref{eq:sys}--\eqref{eq:feedback} in $L^1$ and $L^{\infty}$ norms, according to the following definition.
\begin{definition}
\label{def:expstab}
The system \eqref{eq:sys}--\eqref{eq:feedback} is said globally exponentially stable in the $L^{1}$-norm (resp. in the $L^{\infty}$-norm), if there exist $C>0$ and $\gamma>0$ such that for any $u_0\in L^\infty(0,1)$ and for any $T>0$, there exists a unique entropy solution $u$ of \eqref{eq:sys}--\eqref{eq:feedback} with initial condition $u_{0}$ in the sense of Definition \ref{def:kru} with
\begin{itemize}
  \item $u\in C([0,T];L^1(0,1))\cap L^\infty((0,T)\times(0,1))$,
  \item $u$ admits strong traces $u(\cdot,0),u(\cdot,1)\in L^\infty(0,T)$ which satisfy the boundary condition \eqref{eq:feedback} a.e. on $(0,T)$,
\end{itemize}
and
\begin{equation}
\| u(t,\cdot)\|_{L^{1}(0,1)}\leq C e^{-\gamma t}\|u_{0}\|_{L^{1}(0,1)}, \;\;\forall t\in[0,+\infty),
\end{equation}
\begin{equation}
\label{exp-decay-infty}
(\text{resp. } \|u(t,\cdot)\|_{L^\infty (0,1)}\leq C e^{-\gamma t}\|u_{0}\|_{L^\infty (0,1)}, \;\;\forall t\in[0,+\infty)).
\end{equation}
\end{definition}
Our main results are the following.
\begin{theorem}[Global exponential stability in the $L^1$-norm]
\label{thm:expstab}
Assume that $f$ is globally Lipschitz and that there exist $\mu>0$ and $p_{i}>0$ such that
\begin{equation}
\label{eq:condstab}
\sum\limits_{i=1}^{n}p_{i}\left(|f_{i}(y_{i})-f_{i}(0)|e^{-\mu}-|f_{i}(G_{i}(y))-f_{i}(0)|\right) \geq 0,\;\;\forall
 y = (y_{1},...,y_{n})^\top \in \mathbb{R}^{n}.
\end{equation}
Then the system \eqref{eq:sys}--\eqref{eq:feedback} is globally exponentially stable for the $L^{1}$ norm (in the sense of Definition \ref{def:expstab})
\end{theorem}
\begin{remark}The condition \eqref{eq:condstab} for \textit{global stability} depends explicitly on the functions $f_{i}$. It is therefore instructive to compare it with the \textit{local stability} conditions previously established in the literature. For instance in \cite{CoronNguyen2015}, the system \eqref{eq:sys}--\eqref{eq:feedback} is shown to be \textit{locally} exponentially stable (for the $W^{2,p}$-norm) if there exists $\Delta \in \mathcal{D}^n$ such that
\begin{equation}
\label{eq:delta0}
|\Delta G'(0)\Delta^{-1}|_{p} <1
\end{equation}
where $\mathcal{D}^n$ denotes the set of diagonal $n \times n$ matrices with strictly positive diagonal entries and $|\cdot|_{p}$ $(p\in[1,+\infty])$ denotes the classical matrix norm\footnote{$
|M|_{p}:=\max\{|Mz|_{p};\; z\in \mathbb{R}^n \text{ such that } |z|_{p}\leq 1\},\;\forall\;p\in[1,+\infty],
$
with, for $z=(z_1,\ldots,z_n)^\top \in \mathbb{R}^n$, $|z|_{p}:=\left(\sum_{i=1}^n |z_i|^p\right)^{1/p}$ if $p\in [1,+\infty)$ and $|z|_{\infty}:=\max\{|z_i|; \, i\in \{1,\ldots, n\}\}$.}. The same condition is found also in \cite[Theorem 1.2]{CoronBV} and \cite[Theorem 3.4]{dus2021} to obtain the stability in the BV-norm. An important point to note is that the condition \eqref{eq:delta0} is equivalent to the existence of $\mu>0$ and $\Delta \in \mathcal{D}^n$ such that
\begin{equation}
\label{eq:delta0-1}
|\Delta G'(0)\Delta^{-1}|_{p} \leq e^{-\mu},
\end{equation}
which is equivalent to
\begin{equation}
\label{eq:delta01}
|\Delta G'(0)z|_{p} \leq  e^{-\mu} |\Delta z|_{p},\quad \forall z\in \mathbb{R}^{n}\setminus\{0\}.
\end{equation}
which, in turn, can be recovered from \eqref{eq:condstab} locally around $y=0$ for $p=1$. Indeed, a first-order expansion of \eqref{eq:condstab} around $y=0$ corresponds to
\begin{equation}
\label{eq:condstab-expansion}
\sum_{i=1}^{n} p_{i}|f'_{i}(0)G_{i}'(0)y| \leq \sum_{i=1}^{n}p_{i}|f_{i}'(0)y_i|e^{-\mu},\quad \forall y = (y_{1},\dots,y_{n})^\top \in \mathbb{R}^{n}.
\end{equation}
For each $i\in \{1,\ldots,n\}$, we know that $p_i f'_i(0)$ is strictly positive by \eqref{eq:pos}. Defining the diagonal matrix $\Delta := \mathrm{diag}(p_1 f'_1(0),\ldots, p_n f'_n(0)) \in \mathcal{D}^n$, we can rewrite \eqref{eq:condstab-expansion} as
\begin{equation}
\label{eq:condstab-expansion-new}
|\Delta G'(0)y|_{1} \leq e^{-\mu} |\Delta y|_{1},\quad \forall y \in \mathbb{R}^{n},
\end{equation}
which directly implies \eqref{eq:delta01} for $p=1$ with $\Delta$ depending on the $f'_i(0)$. Hence the condition \eqref{eq:condstab} for global stability in $L^1$ corresponds around $y=0$ to the condition \eqref{eq:delta0} for the local stability in $W^{2,1}$.
\end{remark}

Let us now come back to the presentation of our main results.  First, under a convexity/concavity assumption on the $f_{i}$, a direct global extension of \eqref{eq:condstab-expansion-new} gives a sufficient condition for the global exponential stability in the $L^{1}$-norm which does not depend on the functions $f_{i}$. This is stated in the following theorem.
\begin{theorem}
\label{thm:fconvex-concave}
Assume that $f$ is globally Lipschitz, that the $f_{i}$ are convex on $(-\infty, 0]$ and concave on $[0,+\infty)$, and that there exist $\mu>0$ and $\Delta \in \mathcal{D}^n$ such that
\begin{equation}
\label{eq:condstab2}
|\Delta G(z)|_{1} \leq e^{-\mu}|\Delta z|_{1},\quad\forall z\in \mathbb{R}^{n}.
\end{equation}
Then the system \eqref{eq:sys}--\eqref{eq:feedback} is globally exponentially stable for the $L^{1}$ norm (in the sense of Definition \ref{def:expstab}).
\end{theorem}
Then, interestingly, this result also holds in the $L^{\infty}$ norm without any convexity/concavity assumption on the $f_{i}$.

\begin{theorem}[Global exponential stability in the $L^\infty$-norm]
\label{thm:expstab-infty}
Assume that there exist $\mu>0$ and $\Delta \in \mathcal{D}^n$ such that
\begin{equation}
\label{Grhoinfty<1}
|\Delta G(z)|_\infty\leq e^{-\mu}|\Delta z|_\infty, \quad \forall z\in \mathbb{R}^n.
\end{equation}
Then the system \eqref{eq:sys}--\eqref{eq:feedback} is globally exponentially stable for the $L^{\infty}$ norm (in the sense of Definition \ref{def:expstab}).
\end{theorem}

\begin{remark}[Linear feedback]
If the feedback law is linear, i.e. if $G(z)=Kz$ for some matrix $K\in \R^{n\times n}$, then \eqref{Grhoinfty<1} is equivalent to
\begin{equation}
\label{eq:feedlin}
\rho_\infty(K)<1,
\end{equation}
where
\begin{equation}
\label{rhoinftyK}
\rho_\infty(K):=\inf\{|\Delta K \Delta^{-1}|_\infty;\; \Delta \in \mathcal{D}^n \}.
\end{equation}
Note that it is proved in \cite{1984-Greenberg-Li-JDE}, \cite{1985-Qin-CAM}, \cite{1986-Zhao-thesis}, \cite[Theorem 1.3]{1994-Li-book}, \cite{2015-Bastin-Coron-SICON}, and \cite[Section 4.1]{2016-Bastin-Coron-book} that $\rho_\infty(G'(0))<1$ is a sufficient condition for the local exponential stability of the system \eqref{eq:sys}--\eqref{eq:feedback} in the $C^k$-norm with $k\geq 1$. It is also proved in \cite{dus2021} that the condition \eqref{eq:feedlin} ensures a semi-global exponential stability in the BV norm.
\end{remark}

\begin{remark}
Note that, in contrast with the well-posedness result of Theorem \ref{thm:main}, the global Lipschitzness of $f$ is not needed for the exponential stability results of Theorem~\ref{thm:expstab-infty}. This is allowed thanks to the assumption \eqref{Grhoinfty<1} on $G$, which implies that the system \eqref{eq:sys}--\eqref{eq:feedback} is well-posed for all time $T>0$.
\end{remark}

In many physical cases, the positivity condition \eqref{eq:pos} is too strong (one can think of Burgers, LWR or Buckley-Leverett equations, for instance). Let us however remark that this assumption can be removed if we look at the local exponential stability. Indeed a direct consequence of Theorem \ref{thm:expstab-infty} is the following Corollary.

\begin{corollary}[Local stability in $L^{\infty}$]
\label{cor:linear}
Assume that $f$ satisfies
\begin{equation}
f_{i}'(0)>0, \;\forall i\in\{1,...,n\},
\end{equation}
instead of \eqref{eq:pos} and that
\begin{equation}
\label{eq:feedlinG}
\rho_\infty(G'(0))<1,
\end{equation}
then the system \eqref{eq:sys}--\eqref{eq:feedback} is (locally) exponentially stable for the $L^{\infty}$ norm. Moreover, the exponential estimate holds for any decay rate $\mu$ satisfying
\begin{equation}
\mu<\min\limits_{i\in\{1,..,n\}} f_{i}'(0)\left[-\ln(\rho_{\infty}(G'(0)))\right].
\end{equation}
\end{corollary}

\newpage

\section{Entropy solutions and boundary traces}\label{sec:entropy}

We recall the definition of an entropy solution.

\begin{definition}[Kru\v{z}kov entropy inequalities in the interior \cite{Kruzhkov1970}]\label{def:kru}
A function $u\in L^\infty((0,T)\times(0,1))$ is an \emph{entropy solution of \eqref{eq:sys} in the interior} if for every $k\in\R$ and every nonnegative
$\varphi\in C_c^1([0,T)\times(0,1))$,
\begin{multline}\label{eq:kru}
  \int_{0}^{T}\int_{0}^{1} \Bigl(|u_{i}-k|\varphi_t + \sgn(u_{i}-k)\bigl(f_{i}(u_{i})-f_{i}(k)\bigr)\varphi_x\Bigr)\,dx\,dt \\
  + \int_0^1 |u_{0,i}(x)-k|\varphi(0,x)\,dx \ge 0,\;\;\forall\; i\in\{1,...,n\}.
\end{multline}
\end{definition}
\begin{remark}
\label{rmk:Kruz}
If $u\in L^\infty((0,T)\times(0,1))$ is an entropy solution of \eqref{eq:sys} in the interior, then one has for any convex $\eta$ (see \cite{Kruzhkov1970})
of class $C^1$,
\begin{equation}\label{eq:kru-convex}
  \int_{0}^{T}\int_{0}^{1} \Bigl(\eta(u_i)\varphi_t + q_i(u_i) \varphi_x\Bigr)\,dx\,dt +\int_{0}^{1}\eta(u_{0,i})\varphi(0,x)dx  \ge 0,\;\;\forall\; i\in\{1,...,n\},
\end{equation}
 for every function $\varphi\in C^1_c([0,T)\times (0,1);[0,+\infty))$ provided that
\begin{equation}
\label{def-flux}
q_i'(z)=\eta'(z)f'_i(z), \quad \forall z\in \mathbb{R}.
\end{equation}
The function $q_i$ is called the entropy flux associated to the entropy $\eta$.
\end{remark}

The boundary traces of an entropy solution are defined in the following proposition.

\begin{proposition}[Strong traces]\label{prop:traces}
Let $u\in L^\infty((0,T)\times(0,1))$ be an entropy solution in the sense of Definition~\ref{def:kru}.
Then there exist (strong) boundary traces $u(\cdot,0),u(\cdot,1)\in L^\infty(0,T)$ such that (up to redefining $u$ on a null set),
\begin{equation}
\label{eq:convbords}
  \lim_{\varepsilon \rightarrow 0,\;\varepsilon >0}\int_0^T |u(t,\varepsilon)-u(t,0)|\,dt = 0,
  \qquad
  \lim_{\varepsilon\rightarrow 0,\; \varepsilon >0}\int_0^T |u(t,1-\varepsilon)-u(t,1)|\,dt = 0.
\end{equation}
\end{proposition}

By virtue of \eqref{eq:pos}, Proposition~\ref{prop:traces} follows from \cite[Theorem 1]{KV07} (where $\eta(u):=u$ and $H(u):=f_i(u)$) or \cite[Theorem 1.4]{Panov07} (where $k$ is chosen large enough so that $\psi_k(u)=f_i(k)-f_i(u)$). For both approaches, the natural open set is $\Omega := (0,+\infty) \times (0,1)$. However, $\Omega$ does not satisfy the boundary regularity conditions required by these theorems, namely $C^2$ for \cite{KV07} and $C^1$ for \cite{Panov07}. To resolve this, it suffices to smooth the corners to obtain a sequence of regularized domains $\Omega_l$ that appropriately converge to $\Omega$. Applying the theorems to $\Omega_l$ yields Proposition~\ref{prop:traces}. We omit the explicit details of this approximation. Finally, because the problem is spatially one-dimensional, an alternative approach is to exchange the variables $t$ and $x$, allowing for the application of a trace theorem at $t=0$, such as \cite[Theorem 1.2]{2005-Panov-JHDE}. See also \cite[Theorem 2.4]{2022-Dovetta-Marconi-Spinolo} for this particular framework.

\section{Well-posedness via a step-by-step method}\label{sec:closed}

This section is devoted to the proof of Theorem~\ref{thm:main}. This proof relies on the following:
\begin{enumerate}
    \item The analysis of the open-loop system (Theorem~\ref{thm:open} and Proposition~\ref{prop:mp} below).
    \item Assumption~\eqref{eq:pos}, which guarantees that information propagates to the right at a finite speed (Lemma~\ref{lem:outflow_indep} below). This allows us to treat the \emph{closed-loop} system as a sequence of \emph{open-loop} systems, to which the above analysis can be applied iteratively.
\end{enumerate}
Let us start with the analysis of the  \emph{open-loop} system with a prescribed boundary datum $w$:
\begin{equation}
\label{eq:open}
\begin{cases}
u_t+(f(u))_x=0 & \text{in } (0,T)\times(0,1),\\
u(0,x)=u_0(x) & \text{for a.e. }x\in(0,1),\\
u(t,0)=w(t) & \text{for a.e. }t\in(0,T).
\end{cases}
\end{equation}
Note that imposing the boundary condition $u(t,0)=w(t)$ a priori makes sense under the positivity condition \eqref{eq:pos} on the direction of the characteristics. See Bardos--Le Roux--N\'ed\'elec (BLN) \cite{BLN79} and modern treatments such as \cite{CKK09,AS15} in more general cases.
\begin{theorem}[Open-loop well-posedness]\label{thm:open}
Assume $u_0\in L^\infty(0,1)$, $w\in L^\infty(0,T)$ and $f\in C^1(\R^{n})$.
Then there exists a unique entropy solution $u$ to \eqref{eq:open} such that
$u\in C([0,T];L^1(0,1))\cap L^\infty((0,T)\times(0,1))$ and $u$ admits strong boundary traces $u(t,0) = w(t)$ and $u(t,1)$.
Moreover, for two data $(u_0,w)$ and $(v_0,z)$ with corresponding solutions $u,v$ one has the $L^1$ stability estimate
\begin{multline}\label{eq:L1stab}
  \|u_{i}(t,\cdot)-v_{i}(t,\cdot)\|_{L^1(0,1)}
  \le \|u_{0,i}-v_{0,i}\|_{L^1(0,1)} + \int_0^t |f(w_{i}(s))-f(z_{i}(s))|\,ds,
\\ \forall t\in[0,T],\;\forall\;i\in\{1,...,n\}.
\end{multline}
\end{theorem}
\noindent This theorem is essentially a consequence of \cite[Theorems 1.1--1.2]{CKK09} and \cite[Proposition 2.6]{2022-Dovetta-Marconi-Spinolo}. Its proof is given in Appendix~\ref{app:wellposedopen}.

\begin{proposition}[Maximum principle]
\label{prop:mp}
Under the assumptions of Theorem~\ref{thm:open},
\begin{equation}
\label{eq:propmp}
  \|u\|_{L^\infty((0,T)\times(0,1))} \le \max\Bigl\{\|u_0\|_{L^\infty(0,1)},\|w\|_{L^\infty(0,T)}\Bigr\}.
\end{equation}
\end{proposition}
\noindent This proposition is shown in Appendix \ref{app:wellposedopen}.

\medskip
Our next lemma, proved in Appendix \ref{app:delay}, addresses the finite rightward propagation of information. For the statement of the lemma, we introduce the following notations:
\begin{equation}
\bar{a}(w) = \max\{ f'_{i}(v)\;|\; i\in\{1,...,n\},\; |v|\leq \|u_{0}\|_{L^\infty (0,1)}+\|w\|_{L^\infty (0,1)}\},
\end{equation}
\begin{equation}\label{eq:delta}
  \delta(w) \coloneqq \frac{1}{\bar{a}(w)}.
\end{equation}
\begin{lemma}[Output independent of input for short times]\label{lem:outflow_indep}
Let $u^w,u^z$ be the open-loop entropy solutions of \eqref{eq:open} with the \emph{same} initial datum $u_0\in L^\infty(0,1)$ and two (possibly different) input boundary data
$w,z\in L^\infty(0,T)$.
Then their outgoing traces coincide on $(0,\bar{\delta})$ with $\bar{\delta} = \min(\delta(w),\delta(z))$:
\begin{equation}
  u^w(t,1) = u^z(t,1)
  \qquad\text{for a.e. } t\in(0,\bar{\delta}).
\end{equation}
More generally, if $w=z$ a.e.\ on $(0,\tilde{t})$ for some $\tilde{t}\in(0,T)$, then $u^w(t,1) = u^z(t,1)$ for  a.e. $t$ in $(0,\tilde{t}+\bar{\delta})\cap(0,T)$.
\end{lemma}

Let us now explain how this lemma can be used to reduce the closed-loop system as a sequence of open-loop systems.
First note that when $f$ is globally Lipschitz, then, for any $w\in L^{\infty}$
\begin{equation}
\bar{a}(w) \leq C_{f},
\end{equation}
where $C_{f}$ is the Lipschitz constant of $f$. As a consequence, denoting $\tau = C_{f}^{-1}$,
\begin{equation}
0<\tau \leq \delta(w),\;\forall w\in L^{\infty}.
\end{equation}
Let $T>0$ and $N\in\mathbb{N}$ such that $N\tau > T$ and denote
\begin{equation}t_{n} = \min(n\tau,T).\end{equation}
We inductively construct an input function $w\in L^\infty(0,T)$ and the corresponding open-loop solution $u$ of \eqref{eq:open} such that
$w(t)=G(u(t,1))$ a.e.\ in $(0,T)$.

\medskip
\noindent \textbf{Step $n=1$.}
Choose an arbitrary ``dummy'' input $\widetilde w^{(1)} \in L^\infty(0,T)$ and let $\widetilde u^{(1)}$ be the open-loop solution of \eqref{eq:open} with input $\widetilde w^{(1)}$.
Define
\begin{equation}
  w(t)\coloneqq G(\widetilde u^{(1)}(t,1)) \qquad \text{for a.e. } t\in(0,t_1).
\end{equation}
By Proposition~\ref{prop:traces} and the Lipschitz continuity of $G$, we have $w\in L^\infty(0,t_1)$.

Now let $u^{(1)}$ be the open-loop solution on $(0,t_1)$ with input $w$ (which is unique from Theorem~\ref{thm:open}).
By Lemma~\ref{lem:outflow_indep} (with $t_{0}=0$), the corresponding outgoing traces  $\widetilde u^{(1)}(t,1)$ and $u^{(1)}(t,1)$ coincide a.e.\ on $(0,t_1)$, hence
\begin{equation}
  w(t)=G(u^{(1)}(t,1)) \qquad \text{for a.e. } t\in(0,t_1),
\end{equation}
so the boundary feedback condition \eqref{eq:feedback} is satisfied on $(0,t_{1})$.

\medskip
\noindent\textbf{Induction step.}
Assume that we have defined $w$ on $(0,t_n)$ and constructed $u^{(n)}$, the open-loop solution on $(0,t_n)$ with input $w$, such that
$w(t)=G(u^{(n)}(t,1))$ for a.e.\ $t\in(0,t_n)$.
Choose any extension $\widetilde w^{(n+1)}\in L^\infty(0,T)$ such that $\widetilde w^{(n+1)}=w$ a.e.\ on $(0,t_n)$.

Let $\widetilde u^{(n+1)}$ solve \eqref{eq:open} (on $(0,T)$) with input $\widetilde w^{(n+1)}$.
Define for a.e.\ $t\in(t_n,t_{n+1})$:
\begin{equation}
\label{eq:defwn}
  w(t)\coloneqq G(\widetilde u^{(n+1)}(t,1)).
\end{equation}
Let $\widetilde u^{(n+1)}$ be the open-loop solution on $(0,t_{n+1})$ with this input $w$. By Lemma~\ref{lem:outflow_indep} (applied with $\tau=t_n$), $\widetilde u^{(n+1)}(t,1)$ and $u^{(n+1)}(t,1)$ coincide on $(t_n,t_{n+1})$ since $w$ and $\widetilde{w}^{(n+1)}$ coincide on $(0,t_{n})$ by definition.
Hence,
\begin{equation}w(t)=G(u^{(n+1)}(t,1))\end{equation}
holds a.e.\ on $(t_n,t_{n+1})$.

\medskip
After at most $N$ steps we obtain $w$ on $(0,T)$ and an entropy solution $u$ on $(0,T)\times(0,1)$ of the system \eqref{eq:open} which satisfies the boundary feedback relation \eqref{eq:feedback}. It is therefore an entropy solution of the system \eqref{eq:sys}--\eqref{eq:feedback} since, from
 Theorem~\ref{thm:open}, $u\in C([0,T];L^1(0,1))\cap L^\infty((0,T)\times(0,1))$.

Let us now prove the uniqueness of the solution to the closed-loop system. Let $u,v$ be two entropy solutions of the closed-loop problem \eqref{eq:sys}--\eqref{eq:feedback} with initial condition $u_{0}\in L^{\infty}(0,1)$.
Set $w_u(t)\coloneqq G(u(t,1))$ and $w_v(t)\coloneqq G(v(t,1))$.
On $(0,t_1)$, Lemma~\ref{lem:outflow_indep} (with $\tilde{t}=0$) implies that $u(\cdot,1)=v(\cdot,1)$ a.e.\
 hence
$w_u=w_v$ a.e.\ on $(0,t_1)$.
By open-loop uniqueness (Theorem~\ref{thm:open}), we infer $u=v$ on $(0,t_1)\times(0,1)$.

Assume inductively that $u=v$ on $(0,t_n)\times(0,1)$.
Then $w_u=w_v$ a.e.\ on $(0,t_n)$, and Lemma~\ref{lem:outflow_indep} (with $\tilde{t}=t_n$) yields $u(\cdot,1)=v(\cdot,1)$ a.e.\ on $(t_n,t_{n+1})$.
Thus $w_u=w_v$ a.e.\ on $(t_n,t_{n+1})$, and open-loop uniqueness gives $u=v$ on $(0,t_{n+1})\times(0,1)$.
This ends the proof of the induction and, after finitely many steps, $u \equiv v$ on $(0,T)\times(0,1)$.

This completes the proof of Theorem~\ref{thm:main}.

\section{Exponential stability in the $L^{1}$-norm}
\label{sec:proof-stab-L1}

We now prove Theorem \ref{thm:expstab}. Let $T>0$ and $u$ be an entropy solution of \eqref{eq:sys}--\eqref{eq:feedback}. We define
\begin{equation}
\label{eq:defV}
V(t) = \int_{0}^{1} \sum\limits_{i=1}^{n} p_{i}e^{-\mu x} |u_{i}(t,x)| dx,
\end{equation}
which from Theorem \ref{thm:main} belongs to $C([0,T];\mathbb{R}_{+})$. Let $\phi \in C_{c}^{1}((0,T);\mathbb{R}_{+})$,
\begin{equation}
\int_{0}^{T}V(t)\phi_{t} dt = \sum\limits_{i=1}^{n}\int_{0}^{T}\int_{0}^{1}|u_{i}(t,x)| p_{i}e^{-\mu x}\phi_{t} \,dx\,dt.
\end{equation}
Let us denote
\begin{equation}
W_{i} = \int_{0}^{T}\int_{0}^{1}|u_{i}(t,x)| p_{i}e^{-\mu x}\phi_{t} \,dx\,dt,
\end{equation}
and define
\begin{equation}
\label{eq:wie}
W_{i,\varepsilon} = \int_{0}^{T}\int_{0}^{1}|u_{i}(t,x)| \chi_{\varepsilon}(x)p_{i}e^{-\mu x}\phi_{t} \,dx\,dt,
\end{equation}
where $\chi_{\varepsilon}$ is defined by
\begin{equation}
\label{defchiepsilon}
\chi_{\varepsilon}(x) = \begin{cases}
x/\varepsilon \;\text{ for }x\in(0,\varepsilon]\\
1\; \text{ for }x\in(\varepsilon, 1-\varepsilon]\\
(1-x)/\varepsilon\;\text{ for }x\in(1-\varepsilon,1].
\end{cases}
\end{equation}
Using Definition \ref{def:kru} with $k=0$,
this gives\footnote{\label{foot-note-chi}
Strictly speaking the definition of entropy solutions requires $\varphi\; : (t,x)\mapsto \chi_{\varepsilon}(x)p_{i}e^{-\mu x}\phi(t)$ to belong to $C^{1}_{c}((0,T)\times(0,1);\mathbb{R}_{+})$ and therefore $\chi_{\varepsilon}\in C_{c}^{1}((0,1);\mathbb{R}_{+})$. However the function $\chi_{\varepsilon}$ can be  approximated
 by functions $(\chi_{\varepsilon,k})_{k\in \mathbb{N}}$ in $C_{c}^{1}((0,1);\mathbb{R}_{+})$ such that $|\chi'_{\varepsilon,k}|_{L^\infty(0,1)}\leq 2/\varepsilon$ and $\lim_{k\rightarrow +\infty}\chi'_{\varepsilon,k}(x) =\chi'_{\varepsilon}(x)$, for every $x\in (0,1)\setminus\{\varepsilon, 1-\varepsilon\}$. Then \eqref{eq:kru1} holds with $\chi_{\varepsilon}$ replaced by $\chi_{\varepsilon,k}$ and letting $k \rightarrow +\infty$ in these \eqref{eq:kru1} with $\chi_{\varepsilon,k}$, one gets the desired \eqref{eq:kru1}.}

\begin{equation}
\label{eq:kru1}
\begin{split}
W_{i,\varepsilon} \geq&  - \int_{0}^{T}\int_{0}^{1}\sgn(u_{i})\bigl(f_{i}(u_{i})-f_{i}(0)\bigr)
\left(\chi'_{\varepsilon}(x)-\mu\chi_{\varepsilon}(x)\right)p_{i}e^{-\mu x}\phi(t)\,dx\,dt\\
 =&  -\int_{0}^{T}\left(\frac{1}{\varepsilon}\int_{0}^{\varepsilon}\sgn(u_{i})\bigl(f_{i}(u_{i})-f_{i}(0)\bigr)p_{i}e^{-\mu x}\,dx\right)\,\phi(t)dt\\
 &+\int_{0}^{T}\left(\frac{1}{\varepsilon}\int_{1-\varepsilon}^{1}\sgn(u_{i})\bigl(f_{i}(u_{i})-f_{i}(0)\bigr)p_{i}e^{-\mu x}\,dx\right)\,\phi(t)dt \\
 &+ \int_{0}^{T}\int_{0}^{1}\sgn(u_{i})\bigl(f_{i}(u_{i})-f_{i}(0)\bigr)\mu\chi_{\varepsilon}(x)p_{i}e^{-\mu x}\phi(t)\,dx\,dt.
\end{split}
\end{equation}
Letting $\varepsilon \rightarrow 0$ in the above inequality and in \eqref{eq:wie}, together with \eqref{eq:convbords} we obtain
\begin{multline}
\label{Wigeq}
W_{i} \geq \mu \int_{0}^{T}\int_{0}^{1}\sgn(u_{i})\bigl(f_{i}(u_{i})-f_{i}(0)\bigr)p_{i}e^{-\mu x}\phi(t)\,dx\,dt\\
+ \int_{0}^{T}\Big[\sgn(u_{i}(t,1))\bigl(f_{i}(u_{i}(t,1))-f_{i}(0)\bigr)e^{-\mu}
\\
-\sgn(u_{i}(t,0))
\bigl(f_{i}(u_{i}(t,0))-f_{i}(0)\bigr)\Big]p_{i}\phi(t)dt.
\end{multline}
Using the fact that $u$ satisfies \eqref{eq:feedback} and \eqref{eq:pos}
\begin{multline}
W_{i} \geq  a  \mu \int_{0}^{T}\int_{0}^{1}|u_{i}| p_{i}e^{-\mu x}\phi(t)\,dx\,dt\\
+\int_{0}^{T}\Big(|f_{i}(u_{i}(t,1))-f_{i}(0)|e^{-\mu}-|f_{i}(G_{i}(u(t,1)))-f_{i}(0)|\Big)p_{i}\phi(t)dt,
\end{multline}
where $G_{i}(u)$ is the $i$-th component of $G(u)$ and where we used the fact that, from \eqref{eq:pos}, $f_{i}$ is strictly increasing and therefore $\sign(f_{i}(u)-f_{i}(0)) = \text{sgn}(u)$. Finally, summing and using \eqref{eq:defV} we obtain
\begin{multline}
\int_{0}^{T}V(t)\phi_{t}(t)dt \geq a \mu \int_{0}^{T}V(t)\phi(t) dt \\+ \int_{0}^{T}\sum\limits_{i=1}^{n}\Big( |f_{i}(u_{i}(t,1))-f_{i}(0)|e^{-\mu}-|f_{i}(G_{i}(u(t,1)))-f_{i}(0)|\Big)p_{i}\phi(t)dt.
\end{multline}
Using \eqref{eq:condstab} we obtain
\begin{equation}
\int_{0}^{T}V(t)\phi_{t}(t)dt \geq a \mu \int_{0}^{T}V(t)\phi(t) dt,
\end{equation}
which holds for any $\phi\in C_{c}^{1}((0,T);\mathbb{R}_{+})$ and implies that
\begin{equation}
V(t) \leq V(0)e^{-a \mu t},\forall t\in[0,T].
\end{equation}
From \eqref{eq:defV} we obtain directly that there exists $C$ depending only on $p_{i}$ and $\mu$ such that
\begin{equation}
\|u(t,\cdot)\|_{L^{1}(0,1)} \leq C e^{-a \mu t} \|u_{0}\|_{L^{1}(0,1)},\forall t\in[0,T].
\end{equation}
This completes the proof of Theorem \ref{thm:expstab}.

\medskip
We now prove Theorem~\ref{thm:fconvex-concave}.
We observe that, instead of considering the Lyapunov function candidate \eqref{eq:defV}, we could also consider the more general function
\begin{equation}
V(t) = \int_{0}^{1} \sum\limits_{i=1}^{n} p_{i}e^{-\mu x} \eta_{i}(u_{i}(t,x)) dx,
\end{equation}
where $\eta_{i}$ is a convex function which satisfies
\begin{equation}
\label{eq:enceta}
\alpha |U|\leq \eta_{i}(U)\leq \beta |U|, \quad \forall U\in \mathbb{R},
\end{equation}
where $\alpha$ and $\beta$ are two positive constants. We choose
\begin{equation}
\eta_{i}(U) = \left|\int_{0}^{U}\frac{1}{f_{i}'(s)}ds\right|.
\end{equation}
From \eqref{eq:pos} and the convexity/concavity assumption on the $f_{i}$, it follows that the $\eta_{i}$ are convex. Moreover, by \eqref{eq:pos} and since $f$ is globally Lipschitz, the $\eta_{i}$   satisfy also \eqref{eq:enceta}.
Proceeding as previously, applying \eqref{eq:kru-convex} with entropy $\eta_{i}$ and associated entropy flux $q_{i}(U) = |U|$, for any $\phi\in C^{1}_{c}([0,T);\mathbb{R}_{+})$ we have
\begin{equation}
\int_{0}^{T}V(t)\phi_{t}(t)dt \geq a \mu \int_{0}^{T}V(t)\phi(t) dt + \int_{0}^{T}\sum\limits_{i=1}^{n}\left(|u_{i}(t,1)|e^{-\mu}-|G_{i}(u(t,1))|\right)p_{i}\phi(t)dt.
\end{equation}
Selecting now $p_{i} = \Delta_{i}$, under the assumption \eqref{eq:condstab2}, we have
\begin{equation}
\sum\limits_{i=1}^{n}\left(|u_{i}(t,1)|e^{-\mu}-|G_{i}(u(t,1))|\right)p_{i} \geq 0,
\end{equation}
hence
\begin{equation}
\int_{0}^{T}V(t)\phi_{t}(t)dt \geq a \mu \int_{0}^{T}V(t)\phi(t) dt,
\end{equation}
and similarly as previously the exponential stability in the $L^{1}$ norm holds. This completes the proof of Theorem \ref{thm:fconvex-concave}.

\begin{remark} The weight $e^{-\mu x}$ appearing in \eqref{eq:defV} (see also \eqref{eq:defV-infty} below) is now a standard choice for equations where the transport component is significant. For instance, we refer to the definition of the Lyapunov function $V$ given in \cite{1999-Coron-SICON} for the Euler equations of incompressible fluids. In the context of smooth solutions to 1d hyperbolic systems (or non-smooth solutions to linear hyperbolic systems), one may consult \cite{2002-Xu-Sallet-COCV, 2008-Coron-Bastin-d-Andrea-Novel-SICON, 2016-Bastin-Coron-book}. Furthermore, for entropy solutions of 1d hyperbolic systems or conservation laws, we refer the reader to \cite{2017-Coron-Ervedoza-Ghoshal-Glass-Perrollaz-JDE, 2019-Donadello-Perrollaz-CRAS}.
\end{remark}
\section{Exponential stability in the $L^{\infty}$-norm}
\label{sec:Linfty}
We now prove Theorem~\ref{thm:expstab-infty}. Our proof is partly inspired by \cite{2015-Bastin-Coron-SICON} (see also \cite[Section 4.1]{2016-Bastin-Coron-book}). Let $T>0$ and let
\begin{equation}
\label{property-nu}
\nu\in (0,\mu).
\end{equation}
Let us first assume that  $f$ is globally Lipschitz. Then, according to Theorem \ref{thm:main}, \eqref{eq:sys}--\eqref{eq:feedback} has a unique entropy solution $u$.
Given $\Delta_i>0$ for $i\in \{1,\ldots, n\}$, and letting $\Delta=\mathrm{diag}(\Delta_1,\ldots, \Delta_n)$, we define $V:[0,T]\rightarrow [0,+\infty)$ by (compare with \eqref{eq:defV})
\begin{equation}
\label{eq:defV-infty}
V(t) =\|\left(\Delta_{1}e^{-\nu x} |u_{1}(t,x)|,\ldots,\Delta_{n}e^{-\nu x} |u_{n}(t,x)|  \right)^\top\|_{L^\infty(0,1)},
\end{equation}
which from Theorem \ref{thm:main} belongs to $L^\infty((0,T);\mathbb{R}_{+})$. In contrast with the $V$  of section~\ref{sec:proof-stab-L1}, this new $V$ may be discontinuous. However, since $u\in C([0,T];L^1(0,1))$, $V(t)$  is well defined for every $t\in[0,T]$ and not only for almost every $t\in [0,T]$. In \eqref{eq:defV-infty}, $\mathbb{R}^n$ is equipped with the $|\cdot|_\infty$-norm.
Hence
\begin{equation}
\label{eq:defV-infty-new}
V(t) =\max\{\Delta_{i}\|e^{-\nu x} u_{i}(t,x)\|_{L^\infty(0,1)};\,i\in \{1,\ldots,n\} \}.
\end{equation}
For a (strictly) positive integer $m$, let us define $V_m:[0,T]\rightarrow [0,+\infty)$ by
\begin{equation}
\label{defVp}
V_m(t) := \left(\int_{0}^{1} \sum\limits_{i=1}^{n} \Delta_{i}^{2m}e^{-2m\nu x} u_{i}^{2m}(t,x) \, dx\right)^{\frac{1}{2m}}.
\end{equation}
From Theorem \ref{thm:main} we know that $u\in C([0,T];L^{1}(0,1))\cap L^{\infty}((0,T)\times(0,1))$ from which we gain directly that $u\in C([0,T];L^{2m}(0,1))$ and therefore $V_{m}$
now belongs to $C([0,T];\mathbb{R}_{+})$.

Let us prove that
\begin{gather}
\label{Vmcv-pointwise}
\lim_{m\rightarrow +\infty} V_m(t)=V(t), \quad \forall  t\in (0,T).
\end{gather}
For a fixed $t\in [0,T]$, let $M:[0,1]\rightarrow [0,+\infty)$ be defined by
\begin{equation}
\label{defM}
M(x) := \max\{\Delta_{i}e^{-\nu x} |u_{i}(t,x)|;\; 1\leq i\leq n\}.
\end{equation}
One has
\begin{equation}
\label{M2mleq}
M^{2m}(x) = \max\{  \Delta_{i}^{2m}e^{-2m\nu x} u_{i}^{2m}(t,x) ;\; 1\leq i\leq n\} \leq \sum_{i=1}^{n} \Delta_{i}^{2m}e^{-2m\nu x} u_{i}^{2m}(t,x).
\end{equation}
Integrating \eqref{M2mleq} on $[0,1]$, one gets
\begin{equation}
\label{ML2mleq}
\|M\|_{L^{2m}(0,1)} : = \left( \int_{0}^{1} M^{2m}(x) \, dx \right)^{\frac{1}{2m}} \leq V_m(t).
\end{equation}
For the upper-bound of $V_m$, one has
\begin{equation}
\sum_{i=1}^{n} \Delta_{i}^{2m}e^{-2m\nu x} u_{i}^{2m}(t,x) \leq n \max\{  \Delta_{i}^{2m}e^{-2m\nu x} u_{i}^{2m}(t,x) ;\; 1\leq i\leq n\}= n M^{2m}(x).
\end{equation}
Integrating this inequality on $[0,1]$ and taking the $(2m)$-th root, we get
\begin{equation}
\label{upper-boundVm}
V_m \leq  n^{\frac{1}{2m}} |M|_{L^{2m}(0,1)}.
\end{equation}
Finally \eqref{Vmcv-pointwise} follows from \eqref{ML2mleq}, \eqref{upper-boundVm} and the following classical result
\begin{equation}
\label{cvLmLinfty}
\lim_{m\rightarrow +\infty} \|M\|_{L^{2m}(0,1)}=\|M\|_{L^\infty (0,1)}.
\end{equation}
Let $\phi \in C_{c}^{1}((0,T);\mathbb{R}_{+})$. One has
\begin{equation}
\label{intVmp}
\int_{0}^{T} V_m^{2m}(t) \phi_{t} dt = \sum\limits_{i=1}^{n}W_{m,i}
\end{equation}
with
\begin{equation}
W_{m,i}:= \int_{0}^{T}\int_{0}^{1} \Delta_{i}^{2m} e^{-2m\nu x} u_{i}^{2m}(t,x) \phi_{t}\, dx\, dt.
\end{equation}
Let us define, for $\varepsilon \in (0,1/2]$,
\begin{equation}
\label{eq:wmie}
W_{m,i,\varepsilon}:= \int_{0}^{T}\int_{0}^{1} u_{i}^{2m}(t,x) \chi_{\varepsilon}(x) \Delta_{i}^{2m}e^{-2m\nu x}\phi_{t}\, dx\, dt,
\end{equation}
where $\chi_{\varepsilon}$ is again defined by \eqref{defchiepsilon}. Let $q_{m,i}: \mathbb{R}\rightarrow \mathbb{R}$ be such that
\begin{equation}
\label{def-flux-m-i}
\frac{dq_{m,i}}{dz}(z)=2mz^{2m-1}f'_i(z), \quad \forall z\in \mathbb{R}, \text{ and } q_{m,i}(0)=0.
\end{equation}
Recalling that $u$ is an entropy solution and applying \eqref{eq:kru-convex} with $\eta(z)=z^{2m}$ and $\varphi(t,x)=\phi(t)\chi_{\varepsilon}(x)e^{-2m\nu x}$, one has
\begin{equation}
\label{eq:spliinf}
\begin{split}
W_{m,i,\varepsilon} \geq&  - \int_{0}^{T}\int_{0}^{1}q_{m,i}(u_i(t,x))\left(\chi'_{\varepsilon}(x)-2m\nu\chi_{\varepsilon}(x)\right)
\Delta_{i}^{2m}e^{-2m\nu x}\phi(t)\,dx\,dt\\
 =&  -\int_{0}^{T}\left(\frac{1}{\varepsilon}\int_{0}^{\varepsilon}q_{m,i}(u_i(t,x))\Delta_{i}^{2m}e^{-2m\nu x}\,dx\right)\,\phi(t)dt\\
 &+\int_{0}^{T}\left(\frac{1}{\varepsilon}\int_{1-\varepsilon}^{1}q_{m,i}(u_i(t,x))\Delta_{i}^{2m}e^{-2m\nu x}\,dx\right)\,\phi(t)dt \\
 &+ 2m \nu \int_{0}^{T}\int_{0}^{1}q_{m,i}(u_i(t,x))\Delta_{i}^{2m}e^{-2m\nu x}\chi_{\varepsilon}(x)\phi(t)\,dx\,dt.
\end{split}
\end{equation}
Letting $\varepsilon \rightarrow 0$ in \eqref{eq:spliinf} and in \eqref{eq:wmie} and using \eqref{eq:convbords}, we obtain,
\begin{equation}
\label{Wmigeq}
W_{m,i} \geq 2m \nu \int_{0}^{T}\int_{0}^{1}q_{m,i}(u_i(t,x))\Delta_{i}^{2m}e^{-2m\nu x}\phi(t)\,dx\,dt+\int_0^TB_{m,i}(t)\phi(t)\, dt
\end{equation}
with
\begin{equation}
\label{defBmi}
B_{m,i}(t):=\Delta_{i}^{2m}\left(q_{m,i}(u_i(t,1))e^{-2m\nu}-q_{m,i}(u_{i}(t,0))\right).
\end{equation}
Using the fact that $u$ satisfies \eqref{eq:feedback}, we have
\begin{equation}
\label{Bmi-equal}
B_{m,i}(t):=\Delta_{i}^{2m}\left(q_{m,i}(u_i(t,1))e^{-2m\nu}-q_{m,i}(G_i(u(t,1)))\right)
\end{equation}
where, again, $G_{i}(u)$ is the $i$-th component of $G(u)$. Let $\bar{\mathbf{f}'}>0$ be such that
\begin{equation}
\label{majoration-fi}
f_i'(z) \leq \bar{\mathbf{f}'}, \quad \forall z \in [-\|u_i\|_{L^\infty ((0,T)\times (0,1))}, \|u_i\|_{L^\infty ((0,T)\times (0,1))}],\;\forall i \in \{1,\ldots,n\}.
\end{equation}
 From \eqref{eq:pos}, \eqref{def-flux-m-i}, and \eqref{majoration-fi}, one has
\begin{equation}
\label{encadrementqmi}
 a|\xi|^{2m}\leq q_{m,i}(\xi)\leq \bar{\mathbf{f}'}|\xi|^{2m},\quad \forall
 \xi\in \mathbb{R},\; \forall m \in \mathbb{N}\setminus\{0\},\; \forall  i \in \{1,\ldots,n\}.
\end{equation}
 From \eqref{Bmi-equal} and \eqref{encadrementqmi}, one has
\begin{equation}
\label{estimation-Bm}
B_m(t) \geq
\left(a\left(\sum_{i=1}^{n} \Delta_{i}^{2m} u_i^{2m}(t,1) e^{-2m\nu}\right)- \bar{\mathbf{f}'} \left(\sum_{i=1}^{n}\Delta_{i}^{2m}G_i(u(t,1))^{2m}\right)\right),
\end{equation}
with
\begin{equation}
\label{defBm}
B_m(t):=\sum_{i=1}^{n}B_{m,i}(t).
\end{equation}

Let us prove the existence of $m_0\in \mathbb{N}\setminus\{0\}$ such that
\begin{equation}
\label{Bmgeq0}
B_m(t)\geq 0, \quad \forall t\in [0,T], \; \forall m\geq m_0.
\end{equation}
Let
\begin{gather}
\label{def:alpham}
\alpha_m(t):=\left(
a \left(\sum_{i=1}^{n} \Delta_{i}^{2m} u_i^{2m}(t,1) e^{-2m\nu}\right)\right)^{\frac{1}{2m}},
\\
\label{def:betam}
\beta_m(t):=\left(\bar{\mathbf{f}'} \left(\sum_{i=1}^{n}\Delta_{i}^{2m}G_i(u(t,1))^{2m}\right)\right)^{\frac{1}{2m}}.
\end{gather}
Note that
\begin{equation}
\label{estimate-power1sur2m}
0 \leq \left(\sum_{i=1}^n z_i^{2m}\right)^{\frac{1}{2m}} - |z|_\infty \leq (n^{\frac{1}{2m}} - 1) |z|_\infty, \quad \forall z=(z_1,\ldots z_n)^\top\in \R^n.
\end{equation}
 From \eqref{def:alpham}, \eqref{def:betam}, and \eqref{estimate-power1sur2m}, one has
\begin{gather}
\label{alpham-geq1}
\alpha_m(t)\geq a^{\frac{1}{2m}}e^{-\nu}\max\left\{\Delta_{i} |u_i(t,1)|;\, i\in\left\{1,\ldots,n\right\}\right\},
\\
\label{betamleq1}
\beta_m(t)\leq n^{\frac{1}{2m}}\bar{\mathbf{f}'}^{\frac{1}{2m}}\max\left\{\Delta_{i}
 |G_i(u(t,1))|;\, i\in\left\{1,\ldots,n\right\}\right\}.
\end{gather}
 From \eqref{Grhoinfty<1} and \eqref{betamleq1}
\begin{gather}
\label{betamleq2}
\beta_m(t)\leq n^{\frac{1}{2m}}\bar{\mathbf{f}'}^{\frac{1}{2m}}e^{-\mu}\max\left\{\Delta_{i} |u_i(t,1))|;\, i\in\{1,\ldots,n\right\}.
\end{gather}
Note that
\begin{equation}
\label{limabm}
\lim_{m\rightarrow +\infty} a^{\frac{1}{2m}}=1, \, \lim_{m\rightarrow +\infty} \bar{\mathbf{f}'}^{\frac{1}{2m}}(t)=1, \text{ and } \lim_{m\rightarrow +\infty} n^{\frac{1}{2m}}=1.
\end{equation}
 From \eqref{property-nu}, \eqref{alpham-geq1}, \eqref{betamleq2}, and \eqref{limabm}, we get the existence of $m_0\in \mathbb{N}\setminus\{0\}$ such that
\begin{equation}
\alpha_m(t)\geq\beta_m(t), \quad \forall t\in [0,T], \; \forall m\geq m_0
\end{equation}
a property which, thanks to \eqref{estimation-Bm}, \eqref{def:alpham} and \eqref{def:betam}, implies \eqref{Bmgeq0}.

 From now on we assume that $m\geq m_0$. From \eqref{intVmp}, \eqref{Wmigeq}, \eqref{encadrementqmi} ,  \eqref{defBm}, and \eqref{Bmgeq0}, we get
\begin{equation}
\label{intVmpgeq-new}
\int_{0}^{T} V_m^{2m}(t) \phi_{t} dt \geq  2m\nu a \sum_{i=1}^{n}  \int_{0}^{T}\int_{0}^{1}  \Delta_{i}^{2m} e^{-2m\nu x}  u_{i}^{2m}(t,x) \phi(t)\,dx\,dt.
\end{equation}
Now from \eqref{intVmpgeq-new}, for every $\phi \in C_{c}^{1}((0,T);\mathbb{R}_{+})$ and for every  $0\leq t_1\leq t_2\leq T$, it can be checked that
\begin{equation}
\label{intVmpgeq}
V_m^{2m}(t_2) -V_m^{2m}(t_1) \leq   -2m\nu a  \int_{t_1}^{t_2}V_m^{2m}(t)\,dt.
\end{equation}
Indeed, one can proceed as for the proof of \eqref{Wigeq} (see in particular the footnote of page \pageref{foot-note-chi}): with a standard approximation procedure, one gets that, for every $\varepsilon \in (0,(t_2-t_1)/2]$, \eqref{intVmpgeq-new} holds for
\begin{equation}
\label{defphiepsilon}
\\ \phi(t)=\phi_{\varepsilon}(t) = \begin{cases}
0\;\text{ for } t\in(0,t_1],\\
(t-t_1)/\varepsilon\; \text{ for } t\in(t_1, t_1+\varepsilon],\\
1\;\text{ for } t \in (t_1+\varepsilon, t_2-\varepsilon),\\
(t_2-t)/\varepsilon \;\text{ for } t \in [t_2-\varepsilon, t_2],\\
0\;\text{ for } t \in (t_2, T],
\end{cases}
\end{equation}
and, letting $\varepsilon\rightarrow 0^+$ in \eqref{intVmpgeq-new} with this $\phi$, one gets \eqref{intVmpgeq} if $t_2>t_1$, while \eqref{intVmpgeq} is trivial if $t_1=t_2$.
Let us point out that \eqref{intVmpgeq} implies that
\begin{equation}
\label{nonincreasing}
\text{$V_m$ is a non-increasing function on $[0,T]$.}
\end{equation}
Note that, by the convexity of $\xi\in \mathbb{R}\rightarrow \xi^{2m}$,
\begin{equation}
\label{Vmdecreasing}
V_m^{2m}(t_2) -V_m^{2m}(t_1) \geq  2mV_m^{2m-1}(t_1)\left(V_m(t_2)-V_m(t_1)\right).
\end{equation}
From \eqref{intVmpgeq}, \eqref{Vmdecreasing}, and \eqref{nonincreasing}, one gets
\begin{equation}
\label{Vmdecreasing-est}
V_m^{2m-1}(t_1) \left(V_m(t_2)-V_m(t_1)\right)\leq -\nu a \int_{t_1}^{t_2}V^{2m}_m(t)\, dt\leq -\nu a \left(t_2-t_1\right) V_m^{2m}(t_2).
\end{equation}
One has the following lemma, whose proof is given in Appendix~\ref{appendix-proof-lemmatheta}.
\begin{lemma}
\label{lemma-est-decreasing}
Let $\theta :[0,T]\rightarrow [0,+\infty)$ be a continuous non-increasing function such that, for some positive integer $m\geq 1$ and for some constant $c>0$,
\begin{equation}
\label{property-theta}
\theta^{2m-1}(t_1) \left(\theta(t_2)-\theta(t_1)\right) \leq -c(t_2-t_1) \theta^{2m}(t_2),\quad \forall \;0\leq t_1\leq t_2\leq T.
\end{equation}
Then,
\begin{equation}
\label{distribution-inequality}
\theta(t) \leq e^{-ct}\theta(0), \quad \forall t \in [0,T].
\end{equation}
\end{lemma}
By \eqref{nonincreasing} and \eqref{Vmdecreasing-est}, the assumptions of this lemma are satisfied for $\theta:=V_m$ and $c:=\nu a$. Hence, by this lemma,
\begin{equation}
\label{decreasing-estimatem}
V_m(t) \leq e^{-\nu at}V_m(0),\quad \forall t \in [0,T],\quad \forall m \geq m_{0}.
\end{equation}
Letting $m\rightarrow +\infty$ in \eqref{decreasing-estimatem} and using \eqref{Vmcv-pointwise}, we get
\begin{equation}
\label{decreasing-estimate}
V(t) \leq e^{-\nu at}V(0),\quad \forall t \in [0,T],
\end{equation}
which concludes the proof of Theorem~\ref{thm:expstab-infty} if $f$ is globally Lipschitz  since, by \eqref{eq:defV-infty-new}, there exists $C\geq 1$ independent of $u$, $t$, and $T$ such that
\begin{equation}
\label{equiv-o-norm}
\frac{1}{C}\|u(t,\cdot)\|_{L^\infty (0,1)}\leq V(t)\leq C \|u(t,\cdot)\|_{L^\infty (0,1)}.
\end{equation}

If $f$ is not globally Lipschitz, one can no longer get the existence  of the solution $u$ of \eqref{eq:sys}--\eqref{eq:feedback} as a direct application of  Theorem~\ref{thm:main}. To show, without this assumption of global Lipschitzness, the existence and uniqueness of the solution $u$    (and that it satisfies the exponential stability estimate) on $[0,+\infty)$, one can proceed as follows. For $u_{0}$ given and $T>0$ arbitrary, let
 $F>0$ to be chosen and define
$\tilde{f}_{i}$ that coincides with $f_{i}$ on $\{y\in\mathbb{R}^{n}\;|\; |y|< F \}$ and is extended on $\mathbb{R}^{n}$ in a globally Lipschitz way such that its Lipschitz constant is
\begin{equation}
C_{\tilde{f}_{i}} \leq 2\sup\limits_{|y|\leq F}\{|\partial_{u}f_{i}(y)|\}.
\end{equation}
The system \eqref{eq:sys}, \eqref{eq:feedback} with $\tilde{f}_{i}$ instead of $f_{i}$ satisfies the assumption of Theorem \ref{thm:main} and there exists a unique entropy solution $u\in L^{\infty}((0,T)\times(0,1))$ which satisfies, from \eqref{exp-decay-infty},
\begin{equation}
\|u\|_{L^{\infty}((0,T)\times(0,1))}\leq C \|u_{0}\|_{L^{\infty}},
\end{equation}
where $C$ only depends on the $\Delta_{i}$. Thus if we choose $F>C \|u_{0}\|_{L^{\infty}}$ (note that $u$ might depend on $F$ but $u_{0}$ does not), then $u$ is also a solution of \eqref{eq:sys}, \eqref{eq:feedback} with $f_{i}$ (indeed, $f_{i}(u(t,x))= \tilde{f}_{i}(u(t,x))$ for a.e. $(t,x)\in(0,T)\times(0,1)$). This shows that \eqref{eq:sys}, \eqref{eq:feedback} admits an entropy solution in $L^{\infty}((0,T)\times(0,1))$ with strong boundary traces and satisfies \eqref{exp-decay-infty} on $[0,T)$ instead of $[0,+\infty)$ where $C$ does not depend on $T$. Extending the solution on $[0,+\infty)$ can be done classically as in \cite{2015-Bastin-Coron-SICON,HayatC12019}.

\begin{remark}
Since \eqref{decreasing-estimate} holds for every $\nu\in (0,\mu)$, letting $\nu\rightarrow\mu^-$ leads to
\begin{equation}
\label{decreasing-estimate-tilde}
\widetilde V(t) \leq e^{-\mu at}\widetilde V(0),\quad \forall t \in [0,T],
\end{equation}
for $\widetilde V$ defined by
\begin{equation}
\label{eq:defV-infty-tilde}
\widetilde V(t) := \Big\| \left(\Delta_{1}e^{-\mu x} |u_{1}(t,x)|,\ldots,\Delta_{n}e^{-\mu x} |u_{n}(t,x)|  \right)^\top \Big\|_{L^\infty(0,1)}.
\end{equation}
\end{remark}

\section{Conclusion}
In this paper, we have established the well-posedness and the global exponential stability of 1D systems of scalar conservation laws closed by a nonlocal boundary feedback. While the existing literature on the boundary stabilization of hyperbolic systems mostly focuses on local stability for classical solutions in regular functional spaces (such as $C^1$ or $H^2$), our work addresses the natural framework of weak entropy solutions, where shocks can form in finite time. A key novelty of our approach are the resulting explicit dissipativity conditions on the feedback map to guarantee global exponential stability in both $L^1$ and $L^\infty$ norms. Notably, we do not require the solutions to have bounded variation (BV functions) and our $L^\infty$ stability framework succeeds in guaranteeing exponential decay without requiring the flux to be globally Lipschitz. We also provide a local exponential stability result in $L^{\infty}$ norm under weaker assumptions. This work paves the way for several subsequent open problems. Among others, future research directions include: Investigating the existence and characterization of forward invariant sets for the closed-loop system under boundary feedback; Removing the diagonal assumption on the fluxes to address general systems of conservation laws that are fully coupled in the interior domain; Relaxing the strict wave speed assumption \eqref{eq:pos} in the global exponential stability to encompass systems where characteristic velocities may vanish or change sign; Extending these well-posedness and boundary stabilization strategies to scalar conservation laws in multi-dimensional spatial domains.

\appendix

\section{Proof of Lemma~\ref{lemma-est-decreasing}}
\label{appendix-proof-lemmatheta}

Since $\theta$ is non-negative and non-increasing, if there exists some $t_0 \in [0, T]$ such that $\theta(t_0) = 0$, then $\theta(t) = 0$ for all $t \in [t_0, T]$. In this case, the inequality $\theta(t) \leq \theta(0)e^{-ct}$ holds trivially for every $t\in [t_0,T]$. Thus, we may assume $\theta(t) > 0$ for every $t\in [0,T)$.

 Fix $t \in [0, T)$ and let $h > 0$ such that $t+h \leq T$. Setting $t_1 = t$ and $t_2 = t+h$ in \eqref{property-theta}, we have:
\begin{equation}
\theta^{2m-1}(t)(\theta(t+h) - \theta(t)) \leq -ch \theta^{2m}(t+h).
\end{equation}
Dividing both sides by $h \theta^{2m-1}(t)$ (which is positive), we obtain:
\begin{equation}
\label{esti-quotient}
\frac{\theta(t+h) - \theta(t)}{h} \leq -c \frac{\theta^{2m}(t+h)}{\theta^{2m-1}(t)}.
\end{equation}
We consider the upper right  Dini derivative, defined as
\begin{equation}
D^+\theta(t) = \limsup_{h \to 0^+} \frac{\theta(t+h) - \theta(t)}{h}.
\end{equation}
Taking the limit as $h \to 0^+$ on the right hand side and using the continuity of $\theta$ as well as \eqref{esti-quotient}, we get
\begin{equation}
\label{Dinithetaleq}
D^+\theta(t) \leq \lim_{h \to 0^+} \left( -c \frac{\theta^{2m}(t+h)}{\theta^{2m-1}(t)} \right) = -c \frac{\theta^{2m}(t)}{\theta^{2m-1}(t)} = -c\theta(t).
\end{equation}

Let us now define the auxiliary function $g(t) = \theta(t)e^{ct}$. We compute the upper right  Dini derivative of $g(t)$:
\begin{equation}
D^+ g(t) := \limsup_{h \to 0^+} \frac{\theta(t+h)e^{c(t+h)} - \theta(t)e^{ct}}{h}.
\end{equation}
Using the expansion $e^{ch} = 1 + ch + o(h)$, we have
\begin{equation}
D^+ g(t) = e^{ct} \limsup_{h \to 0^+} \left[ \frac{\theta(t+h) - \theta(t)}{h} + c\theta(t+h) \right],
\end{equation}
which, with \eqref{Dinithetaleq}, leads to
\begin{equation}
D^+ g(t) \leq e^{ct} [-c\theta(t) + c\theta(t)] = 0.
\end{equation}
Since $g(t)$ is continuous and its upper right  Dini derivative is non-positive, $g(t)$ is non-increasing on $[0, T]$ (see, for example, \cite[Chaper 5, Section 1, Proposition p. 99]{1988-Royden-book}). Therefore, we have $g(t)\leq g(0)$, which gives
\begin{equation}
\theta(t)e^{ct} \leq \theta(0)e^{c(0)} = \theta(0).
\end{equation}
This concludes the proof of Lemma~\ref{lemma-est-decreasing}.
\qed

\section{Proof of Theorem \ref{thm:open} and Proposition \ref{prop:mp}}
\label{app:wellposedopen}
We start with the proof of Theorem \ref{thm:open}. Since the conservation laws of \eqref{eq:open} are decoupled in open-loop, one only needs to establish the result for a single scalar equation. Therefore, in the following, index $i$ is dropped to simplify the notations. For a scalar equation, the existence of a unique entropy solution in the sense of Definition \ref{def:kru} which has strong boundary traces is given in \cite[Theorems 1.1--1.2]{CKK09}. In fact \cite[Theorems 1.1--1.2]{CKK09} uses a non-degeneracy assumption that is not assumed here. However, here the situation is simpler as the domain is one-dimensional and $f'>0$ on the domain  (see also \cite[Proposition 2.6]{2022-Dovetta-Marconi-Spinolo} for this particular framework). As a consequence, the trace on the left in the sense of \cite[Theorems 1.1--1.2]{CKK09} actually implies that the left datum is taken literally (for a.e. $t$, the trace $u(t,0)$ equals the prescribed Dirichlet datum $w(t)$). Deriving the estimate is relatively classical: using a doubling of variable \cite[Proof of Theorem 1; (3.2)]{Kruzhkov1970} and from the definition of entropy solutions one obtains the following, for any $\phi\in C^{1}_{c}([0,T);\mathbb{R}_{+})$ and
$\chi\in \{C^{1}((0,1);\mathbb{R}_{+})\;|\; \chi (0)=\chi(1)=0\}$,
\begin{multline}
\label{eq:defopenth}
\int_{0}^{T}\int_{0}^{1} \Big( |u(t,x)-v(t,x)|\phi'(t)\chi(x) + \text{sgn}(u(t,x)-v(t,x))
\\ [f(u(t,x))-f(v(t,x))]\chi'(x)\phi(t) \Big) dxdt
+\int_{0}^{1}|u_{0}-v_{0}|\chi(x)\phi(0) dx\geq 0.
\end{multline}
Note that equation \eqref{eq:defopenth} also holds when $\chi$ is replaced by $\chi_{\varepsilon}$, defined in \eqref{defchiepsilon} (see \eqref{eq:kru1}).
We then obtain by taking $\varepsilon\rightarrow 0$ and using the existence of strong traces of the solution  together with \eqref{eq:open} (see also \eqref{eq:kru1}--\eqref{Wigeq})
 \begin{equation}
 \begin{split}
&\int_{0}^{T}\int_{0}^{1}|u(t,x)-v(t,x)|\phi'(t)\,dx\,dt+\int_{0}^{1}|u_{0}-v_{0}|\phi(0) dx\\
& -\int_{0}^{T} \text{sgn}(u(t,1)-v(t,1))(f(u(t,1))-f(v(t,1)))\phi(t)dt\\
&+\int_{0}^{T} \text{sgn}(w(t)-z(t))(f(w(t))-f(z(t)))\phi(t)dt
\geq 0.
\end{split}
 \end{equation}
Since $f$ is non-decreasing, we get
\begin{multline}
\int_{0}^{T}\int_{0}^{1}|u(t,x)-v(t,x)|\phi'(t)\,dx\,dt+\int_{0}^{1}|u_{0}-v_{0}|\phi(0) dx
\\
+\int_{0}^{T} |f(w(t))-f(z(t))|\phi(t)dt \geq 0.
\end{multline}
 For $t$ given, choosing $\phi$ as an approximation of $\mathbf{1}_{[0,t]}$, the indicator function of $[0,t]$ and letting it converge to $\mathbf{1}_{[0,t]}$, we have exactly \eqref{eq:L1stab}. This ends the proof of Theorem \ref{thm:open}.

For Proposition \ref{prop:mp}, one only needs to note that, thanks to Remark \ref{rmk:Kruz}, one can choose the entropy $\eta: z\rightarrow (z-k)_{+}:=\max(z-k,0)$ for any $k\in\mathbb{R}$. The associated entropy flux  is $q: z\rightarrow (f(z)-f(k))_{+}$ because $f$ is increasing. Then, proceeding as above, we obtain, for any $k\in\mathbb{R}$ and for any $\phi\in C_{c}^{1}([0,T);\mathbb{R}_{+})$,

 \begin{multline}
\int_{0}^{T}\int_{0}^{1}(u(t,x)-k)_{+}\phi'(t)\,dx\,dt+\int_{0}^{1}(u_{0}(x) - k)_{+}\phi(0) dx
\\+\int_{0}^{T} (f(w(t))-f(k))_{+}\phi(t)dt \geq 0.
\end{multline}
Thus, for any $t\in(0,T)$, choosing $\phi$ as an approximation of the characteristic function of $[0,t)$, we obtain
\begin{equation}
\int_{0}^{1}(u(t,x)-k)_{+}dx \leq \int_{0}^{1}(u_{0}(x)-k)_{+}dx + \int_{0}^{t} (f(w(s))-f(k))_{+}dt.
\end{equation}
Now, assume that \eqref{eq:propmp} does not hold, then one can choose\footnote{Again, this entropy is not $C^{1}$, however \eqref{eq:kru-convex} still hold for continuous, convex and piecewise $C^{1}$ functions by approaching them weakly by convex $C^{1}$ functions.} $$k\in \left(\max\Bigl\{\|u_0\|_{L^\infty(0,1)},\|w\|_{L^\infty(0,T)}\Bigr\},\|u\|_{L^{\infty}((0,T)\times(0,1))}\right)$$ and get that
\begin{equation}
\int_{0}^{1}(u(t,x)-k)_{+}dx \leq 0,
\end{equation}
which implies that $u(t,x) \leq k$ a.e. on $(0,T)\times (0,1)$. The same can be done with $(-k-u(t,x))_{+}$ to deduce that $-u(t,x) \leq k$ and this contradicts the fact that $k<\|u\|_{L^{\infty}((0,T)\times(0,1))}$.

\section{Proof of Lemma~\ref{lem:outflow_indep}}
\label{app:delay}
\begin{proof}
Lemma~\ref{lem:outflow_indep}
 is a finite speed of propagation statement: data at the input boundary $x=0$ cannot influence a neighborhood of $x=1$ before the travel time $\bar{\delta}$. Let us deal with the scalar case, and since \eqref{eq:open} is diagonal the non-scalar case follows directly. Let $I(t) := (\bar{\delta}^{-1} (t-\tilde{t}),1)$ defined for $t\in[0,\bar{\delta})$, we want to study
\begin{equation}
E(t) := \int_{I(t)\cap(0,1)} |u^{w}(t,x)-u^{z}(t,x)|dx.
\end{equation}
Note that $E$ is continuous since $u^{w}$ and $u^{z}$ belong to $C([0,T];L^{1}(\mathbb{R}))$. Using the same doubling of variables as in Appendix \ref{app:wellposedopen} (see also \cite[Theorem 1, (3.2)]{Kruzhkov1970}) but with $\chi_{\varepsilon}(t,x)$ defined by
\begin{equation}
\chi_{\varepsilon}(t,x) :=\begin{cases}
0\;\text{ for }x\in[0,\bar{\delta}^{-1}(t-\tilde{t})_{+}],\\
 (x-\bar{\delta}^{-1}(t-\tilde{t})_{+} )/\varepsilon \; \text{ for }x\in[\bar{\delta}^{-1}(t-\tilde{t})_{+},\bar{\delta}^{-1}(t-\tilde{t})_{+}+\varepsilon],\\
1\;\text{ for }x\in[\bar{\delta}^{-1}(t-\tilde{t})_{+}+\varepsilon,1-\varepsilon),\\
(1-x)/\varepsilon\text{ otherwise},
 \end{cases}
 \end{equation}
 one has
\begin{multline}\label{eq:dissipation}
\int_{0}^{T}\int_{\bar{\delta}^{-1}(t-\widetilde{t})_{+}}^{1}|u^{w}(t,x)-u^{z}(t,x)|\phi'(t)\,dx\,dt
\\- \bar{\delta}^{-1}\int_{0}^{T}|u^{w}(t,\bar{\delta}^{-1}(t-\tilde{t})_{+})-u^{z}(t,\bar{\delta}^{-1}(t-\tilde{t})_{+})|\phi(t)dt\\
-\int_{0}^{T} \text{sgn}(u^{w}(t,1)-v^{z}(t,1))(f(u^{w}(t,1))-f(u^{z}(t,1)))\phi(t)dt\\
+\int_{0}^{T} \text{sgn}(u^{w}(t,\bar{\delta}^{-1}(t-\tilde{t})_{+})-u^{z}(t,\bar{\delta}^{-1}(t-\tilde{t})_{+})\\
(f(u^{w}(t,\bar{\delta}^{-1}(t-\tilde{t})_{+}))-f(u^{z}(t,\bar{\delta}^{-1}(t-\tilde{t})_{+})))\phi(t)dt
\geq 0.
\end{multline}
where we used that $\partial_{t}\chi_{\varepsilon}(t,x) = -(\bar{\delta}\varepsilon)^{-1}\mathbf{1}_{\left(\bar{\delta}^{-1}(t-\tilde{t})_{+},\bar{\delta}^{-1}(t-\tilde{t})_{+}+\varepsilon \right)}(x)\mathbf{1}_{(\tilde{t},T)}(t)$ (here $\mathbf{1}_{[a,b]}$ refers to the indicator function of $[a,b]$) and
\begin{multline}
\lim\limits_{\varepsilon\rightarrow 0} -\frac{\bar{\delta}^{-1}}{\varepsilon}\mathbf{1}_{t>\tilde{t}} \int_{0}^{T}\int_{\bar{\delta}^{-1}(t-\tilde{t})_{+}}^{\bar{\delta}^{-1}(t-\tilde{t})_{+}+\varepsilon}|u^{w}(t,x)-u^{z}(t,x)|\phi(t)dxdt =
\\
- \bar{\delta}^{-1}\int_{0}^{T}|u^{w}(t,\bar{\delta}^{-1}(t-\tilde{t})_{+})-u^{z}(t,\bar{\delta}^{-1}(t-\tilde{t})_{+})|\phi(t)dt.
\end{multline}
Using now that $\bar{\delta}^{-1}>|f'(s)|$ for any $s\in[-M,M]$ with $$M:=\max\{\|u^{w}\|_{L^{\infty}((0,T)\times(0,1))},\|u^{z}\|_{L^{\infty}((0,T)\times(0,1))}\}$$ (see Proposition \ref{prop:mp} and the definition of $\bar{\delta}$ given by \eqref{eq:delta}), we have
\begin{multline}
\label{eq:conseq0}
\int_{\tilde{t}}^{T}\left[\text{sgn}(u^{w}(t,\bar{\delta}^{-1}(t-\tilde{t})_{+})-u^{z}(t,\bar{\delta}^{-1}(t-\tilde{t})_{+})\right.
\\ (f(u^{w}(t,\bar{\delta}^{-1}(t-\tilde{t})_{+}))-f(u^{z}(t,\bar{\delta}^{-1}(t-\tilde{t})_{+})))\\
\left.-\bar{\delta}^{-1}|u^{w}(t,\bar{\delta}^{-1}(t-\tilde{t})_{+})-u^{z}(t,\bar{\delta}^{-1}(t-\tilde{t})_{+})|\right]\phi(t) dt \leq 0,
\end{multline}
and, since $w=z$ a.e. on $(0,\tilde{t})$ by assumption, and using \eqref{eq:open}
\begin{equation}
\label{eq:conseq1}
\int_{0}^{\tilde{t}}\text{sgn}(u^{w}(t,0)-u^{z}(t,0))(f(u^{w}(t,0))-f(u^{z}(t,0)))\phi(t) dt = 0.
\end{equation}
As a consequence from \eqref{eq:conseq0}, \eqref{eq:conseq1}, choosing $\phi$ as an approximation of $\mathbf{1}_{[0,t)}$ in \eqref{eq:dissipation}, we have
\begin{multline}\label{eq:dissipation-new}
\int_{\bar{\delta}^{-1}(t-\tilde{t})}^{1}|u^{w}(t,x)-u^{z}(t,x)|dx
\\
\leq -\int_{0}^{t} \text{sgn}(u^{w}(t,1)-v^{z}(t,1))(f(u^{w}(t,1))-f(u^{z}(t,1)))dt\leq 0.
\end{multline}
Hence, $u^{w} = u^{z}$ a.e. on $(\bar{\delta}^{-1}(t-\tilde{t})_{+},1)$ as long as $t<\tilde{t}+\bar{\delta}$ and since both solutions admit strong traces at $x=1$ we deduce that $u^{w}(t,1)= u^{z}(t,1)$ and this holds for a.e. $t\in(0,\tilde{t}+\bar{\delta})$.
\end{proof}

\section*{Acknowledgements}
The authors would like to thank the two anonymous reviewers for their useful comments. This work was supported by the ANR-Tremplin StarPDE (ANR-24-ERCS-0010) and the Hi!Paris Chair DESCARTES.

\vspace{1cm}

\noindent \textbf{Ethics and Consent to Participate declarations:} Not applicable.

\noindent \textbf{Author Contribution declaration:} All three authors contributed equally to this manuscript.

\noindent \textbf{Data Availability declaration:} Not applicable.

\noindent \textbf{Competing Interest declaration:} The three authors declare that they have no competing financial interests or
personal relationships that could have appeared to influence the work reported
in this manuscript.

\bibliographystyle{plain}
\bibliography{biblio}

\end{document}